\documentclass[12pt]{amsart}
\usepackage{amssymb}
\usepackage[all]{xy}
\usepackage{amsmath}
\usepackage{cases}
\usepackage{amsfonts}
\usepackage{graphicx}
\usepackage{amssymb}
\usepackage{bbm}

\usepackage{xcolor}
\usepackage[colorlinks,linkcolor=blue,citecolor=blue,pagebackref,hypertexnames=false, breaklinks]{hyperref}

\numberwithin{equation}{section}

\newcommand\M{\mathcal{M}}

\usepackage{amsthm}
\newtheorem{thm}{Theorem}[section]
 \newtheorem{cor}[thm]{Corollary}
 \newtheorem{lem}[thm]{Lemma}
 \newtheorem{prop}[thm]{Proposition}
 \theoremstyle{definition} 
 \newtheorem{defn}[thm]{Definition}
 \newtheorem{rem}[thm]{Remark}
 
 \newcommand{\abs}[1]{{\left\lvert{#1}\right\rvert}}
\newcommand{\norm}[1]{{\left\lVert{#1}\right\rVert}}
\newcommand{\br}[1]{{\left({#1}\right)}}

\author{Li Chen}
\address{Li CHEN, Mathematical Sciences Institute, The Australian National University, Canberra ACT 0200, Australia}
\email{li.chen@anu.edu.au}

\begin{document}

\title[Sub-Gaussian heat kernel estimates and quasi Riesz transforms]{Sub-Gaussian heat kernel estimates and quasi Riesz transforms for $1\leq p\leq 2$}
\maketitle

\noindent{\small{{\bf Abstract.} On a complete non-compact Riemannian manifold $M$, we prove that a so-called quasi Riesz 
transform is always $L^p$ bounded for $1<p\leq 2$. If $M$ satisfies the doubling volume property and the sub-Gaussian heat kernel 
estimate, we prove that the quasi Riesz transform is also of weak type $(1,1)$. }} 



%%%%%%%%%%%%%%%%%%%%%%%%%%%%%%%%%%%%%%%%%%%%%%%%%%%%%
\section{Introduction}
Let $M$ be a complete non-compact Riemannian manifold. Let $d$ be the geodesic distance and $\mu$ be the  
Riemannian measure. Denote by $B(x,r)$ the ball of center $x$ and of geodesic radius $r$. We write $V(x,r)=\mu(B(x,r))$.
One says that $M$ satisfies the doubling volume property if, there exists a constant $C>0$ such that for any $x\in M$ and $r>0$,
\begin{equation}
\label{doubling}V(x,2r)\leq CV(x,r).\tag{$D$}
\end{equation}
A simple consequence of (\ref{doubling}) is that there exist $\nu>0$ and $C>0$ such that
\begin{equation} \label{D1}
\frac{V(x,r)}{V(x,s)}\leq C\left(\frac{r}{s}\right)^\nu,
\,\,\forall x\in M, \,r\geq s>0.
\end{equation}

Let $\nabla$ be the Riemannian gradient and $\Delta$ be the non-negative Laplace-Beltrami operator on $M$. By definition and by 
spectral theory, we have 
\[
\int_M |\nabla f|^2  d\mu=\int_M (\Delta f) f  d\mu.=\int_M \br{\Delta^{1/2} f}^2  d\mu,\,\,\forall f \in \mathcal C_0^{\infty}(M).
\]  

It was asked by Strichartz \cite{Str83} in 1983 on which non-compact Riemannian manifold $M$, and for which $p$, 
$1<p<+\infty$, the two semi-norms $\Vert |\nabla f | \Vert _p$ and $\Vert \Delta^{1/2} f\Vert _p$ were equivalent on $
\mathcal{C}_0^{\infty}(M)$. That is, when do there exist two constants $c_p, C_p$ such that
\begin{equation} \label{Ep}\tag{$E_p$}
c_p\Vert \Delta^{1/2}f\Vert _p\leq\Vert |\nabla f |\Vert _p \leq C_p\Vert \Delta^{1/2}f\Vert _p, \,\,\forall f \in \mathcal 
C_0^{\infty}(M)?
\end{equation}

One says that the Riesz transform $\nabla\Delta^{-1/2}$ is $L^p$ bounded on $M$ if 
\begin{equation} \label{Rp} \tag{$R_p$}
\Vert | \nabla f|\Vert _p \leq C\Vert \Delta^{1/2} f\Vert _p, \,\,\forall f \in \mathcal C_0^{\infty}(M).
\end{equation}

Ever since, a lot of work has been dedicated to address the problem, see for example, \cite{B87,CD99,CD03, 
ACDH04,AC05,CCH06,Ca07,CS10,De} and the references therein. 

Denote by $(e^{-t\Delta})_{t>0}$ the heat semigroup associated with $\Delta$ and $p_t(x,y)$ the heat kernel, that is
\[
e^{-t\Delta}f(x)=\int_M p_t(x,y) f(y) d\mu(y),\,\, f\in L^2(M,\mu),\mu\text{-a.e.~}x\in M.
\]
Estimates of the heat kernel and its derivatives happen to be a key ingredient for the boundedness of the Riesz transform. 

Let us first recall a result of Coulhon and Duong in \cite{CD99}.
\begin{thm}\label{thm-cd}
Let $M$ be a complete non-compact manifold satisfying (\ref{doubling}). Assume that 
\begin{equation} \label{DUE} \tag{$DU\!E$}
p_t(x,x)\leq \frac{C}{V(x,\sqrt t)},
\end{equation}
for all $x\in M$, $t>0$ and some $C>0$. Then the Riesz transform $\nabla\Delta^{-1/2}$ is of weak type $(1,1)$ and (\ref{Rp}) 
holds for $1<p\leq2$.
\end{thm}

Under the doubling volume property, (\ref{DUE}) self-improves into the Gaussian heat kernel estimate (see for example \cite{CS08}, 
\cite{Gr09}):
\begin{equation} \label{UE} \tag{$U\!E$}
p_t(x,y)\leq \frac{C}{V(x,\sqrt t)} \exp\left(-c\frac{d^2(x,y)}{t}\right),\,\,\forall x,y\in M, t>0.
\end{equation}

Note that (\ref{Rp}) may be false under (\ref{doubling}) and (\ref{DUE}) for $p> 2$. For example, the connected sum of two copies of
$\mathbb R^n$, $n\geq 2$, does satisfy (\ref{doubling}) and (\ref{DUE}), but the Riesz transform is not $L^p$ bounded for $p>n$. 
We refer to \cite{CD99,CCH06,Ca07} for more details. However, it is not known whether (\ref{DUE}) is necessary for the $L^p$ 
boundedness of the Riesz transform for $1< p< 2$. 

\medskip
We are going to see that one can still obtain a weaker version of (\ref{Rp}) for $1<p\leq2$ without assuming any heat kernel 
estimates.

To this end, we first localise the Riesz transform at infinity. Then one can consider some weaker variants of this localisation.
In fact, we are going to prove that on every complete manifold the Riesz transform is almost bounded in the following sense:
\begin{thm}
\label{thm1} Let $M$ be a complete manifold. Then, for any $\alpha \in (0,1/2)$, the operator $\nabla e^{-\Delta }
\Delta^{-\alpha}$ is bounded on $L^p$ for all $1<p\leq 2$.
\end{thm}

Together with known local results, this yields:
\begin{prop}\label{exo}
Let $M$ be a complete Riemannian manifold satisfying ($D_{loc}$) and ($DU\!E_{loc}$), then the quasi Riesz transform 
$\nabla (I+\Delta)^{-1/2}+\nabla e^{-\Delta}\Delta^{-\alpha}$ with $\alpha \in (0,1/2)$ is $L^p$ bounded for $1<p\leq 2$. Here 
($D_{loc}$) and ($DU\!E_{loc}$) are local versions of (\ref{doubling}) and (\ref{DUE}) to be explained below.
\end{prop}
\begin{rem}
It is also equivalent to say that the operator $\nabla(\Delta^{\alpha}+\Delta^{1/2})^{-1}$ is $L^p$ bounded.
\end{rem}

A natural question would be to ask whether (\ref{Rp}) holds if we replace (\ref{DUE}) with some other kind of heat kernel 
estimates, for instance, the sub-Gaussian heat kernel upper bound introduced in \cite{HSC01,GT12} as follows:
\begin{defn}
We say that the heat kernel on $M$ satisfies the sub-Gaussian upper bound with $m>2$ if for any $x,y \in M$,
\begin{equation}
\label{ue} p_t(x,y)\leq \frac{C}{V\left(x,\rho^{-1}(t)\right)}\exp\left(-c G\br{d(x,y),t}\right),
\tag{$U\!E_m$}
\end{equation}
where
\begin{eqnarray}
\label{t} \rho(t)=\left\{ \begin{aligned}
         &t^{2},&0<t<1, \\
         & t^{m},&t\geq 1;
         \end{aligned}\right.
\end{eqnarray}
and
\begin{eqnarray}
G(r,t)=\left\{ \begin{aligned}
         &\frac{r^2}{t},&t\leq r, \\
         & \br{\frac{r^m}{t}}^{1/(m-1)},&t\geq r.
         \end{aligned}\right.
\end{eqnarray}
\end{defn}
Note that $(U\!E_2) = (U\!E)$. For $m>2$, (\ref{ue}) is neither stronger nor weaker than (\ref{UE}), see Section 3.1 below.
See also Section 3.1 for examples that satisfy (\ref{ue}).

Under (\ref{doubling}) and (\ref{ue}), we are not able to show that the Riesz transform is $L^p$ bounded for $1<p<2$. But we are 
able to treat the endpoint case $p=1$ of Proposition \ref{exo}. Indeed, we have
\begin{thm}\label{weak}
Let $M$ be a complete Riemannian manifold satisfying (\ref{doubling}) and (\ref{ue}). Then for any $0<\alpha<1/2$, the quasi Riesz 
transform $\nabla (I+\Delta)^{-1/2}+\nabla e^{-\Delta} \Delta^{-\alpha}$ is of weak type $(1,1)$.
\end{thm}

Note that one can define Hardy spaces associated with the Laplacian via square functions, which are adapted to the heat kernel 
estimates, and show that $\nabla e^{-\Delta} \Delta^{-\alpha}$ is $H^1-L^1$ bounded, see \cite{Ch}.

\medskip

The plan of this paper is as follows:

In Section 2, we describe the relations between Riesz transform, local Riesz transform, Riesz transform at infinity and quasi 
Riesz transform, and we prove Theorem \ref{thm1}. 

In Section 3, we consider Riemannian manifolds satisfying (\ref{doubling}) and (\ref{ue}). We show Theorem \ref{weak}. 

\medskip
Throughout this paper, we often write $B$ for the ball $B(x_B, r_B)$. For any given $\lambda>0$, we will write $\lambda B$ for 
the $\lambda$ dilated ball,  which is the ball with the same center as $B$ and with radius $r_{\lambda B} = \lambda r_B$. 
We denote $C_1(B)=4B$, and $C_j(B)=2^{j+1} B\setminus 2^{j}B$ for $j=2,3,\cdots$. 

The letters $c,C$ denote positive constants, which can change in different circumstances. We say that $A\lesssim B$ if there 
exists a constant $C>0$ such that $A\leq CB$. And $A\simeq B$ if there exist two positive constants $c,C$ with $c\leq C$
such that $cA\leq B\leq CA$.

\bigskip
%%%%%%%%%%%%%%%%%%%%%%%%%%%%%
%%%%%%%%%%%%%%%%%%%%%%%%%%%%%
\section{$L^p$ boundedness of quasi Riesz transforms}

In this section, unless otherwise stated, we always consider an arbitrary complete Riemannian manifold $M$ 
without any other assumptions.

We could as well consider a metric measure space setting associated with a regular and strongly local Dirichlet form, which admits 
a ``carr\'e du champ" (see \cite{BE85,GSC11}).

%%%%%%%%%%%%%%%%%%%%%%%%%%%%%%%%%%%%%%%%%%%%%%%%%%%%%

\subsection{Localisation of Riesz transforms}
Write the Riesz transform 
\[
\nabla \Delta^{-1/2} =\int_{0}^{\infty} \nabla e^{-t\Delta} \frac{dt}{t^{1/2}}.
\]
Alexopoulos \cite{Al02} separated the integral into local and global parts as $\int_0^1+\int_1^{\infty}$ and considered them
respectively to show the $L^p$ boundedness of the Riesz transform.

An alternative and equivalent method given in \cite{DER03} is to consider the following local Riesz transform and Riesz 
transform at infinity:  

For $1<p<\infty$, we say that the local Riesz transform is $L^p$ bounded if
\begin{equation} \label{Rp-loc} \tag{$R_p^{loc}$}
\norm{\abs{\nabla f}}_p \leq C \norm{(I+\Delta)^{1/2} f}_p, \,\,\forall f \in \mathcal C_0^{\infty}(M).
\end{equation}
and the Riesz transform at infinity is $L^p$ bounded if
\begin{equation} \label{Rp-infty} \tag{$R_p^{\infty}$}
\norm{\abs{\nabla e^{-\Delta} f}}_p \leq C \norm{\Delta^{1/2} f}_p, \,\,\forall f \in \mathcal C_0^{\infty}(M).
\end{equation}

\begin{rem}
Note that at high frequencies, $(I+\Delta)^{-1/2}\simeq \Delta ^{-1/2}$. Thus locally $\nabla(I+\Delta)^{-1/2}$ is the 
Riesz transform. Similarly, since $e^{-\Delta}\Delta^{-1/2} \simeq \Delta^{-1/2}$ when $\Delta \ll \epsilon$  (i.e. 
at low frequencies), we can regard the operator $\nabla e^{-\Delta }\Delta^{-1/2}$ as the localisation of Riesz transform at infinity. 
\end{rem}

A local version of Theorem \ref{thm-cd} says
\begin{thm}[\cite{CD99}]\label{Rploc}
Let $M$ be a complete Riemannian manifold satisfying the local doubling volume property ($D_{loc}$)
\begin{equation} \label{Dloc} \tag{$D_{loc}$}
\forall r_0>0, \exists C_{r_0} \text{ such that } V(x,2r) \leq C_{r_0} V(x,r), \,\,\forall x \in M, r \in (0,r_0),
\end{equation}
and whose volume growth at infinity is at most exponential in the sense that
\[
V(x,\lambda r) \leq C e^{c\lambda} V(x,r),\,\,\forall x\in M, \lambda>1, r\leq1.
\]
Suppose 
\begin{equation} \label{DUEloc} \tag{$DU\!E_{loc}$}
p_t(x,x) \leq \frac{C}{V(x,\sqrt t)},\,\,\forall x\in M, t\in (0,1].
\end{equation}
Then ($R_p^{loc}$) holds for $1<p\leq 2$.
\end{thm}
Examples that satisfy the above assumptions include Riemannian manifolds with Ricci curvature bounded from below.

\medskip
We can characterise the $L^p$ boundedness of Riesz transform by the combination of (\ref{Rp-loc}) and (\ref{Rp-infty}). That is, 
\begin{thm}
Let $M$ be a complete Riemannian manifold. Then, for $1<p<\infty$, the Riesz transform $\nabla \Delta^{-1/2}$ is 
$L^p$ bounded on $M$ if and only if (\ref{Rp-loc}) and (\ref{Rp-infty}) hold.
\end{thm}

The proof relies on the following multiplier theorem due to Cowling:
\begin{thm}[\cite{Cow83}]\label{fc}
Let $M$ be a measure space. Let $L$ be the generator of a bounded analytic semigroup on $L^p(M)$ for $1<p<\infty$ such that 
$e^{-tL}$ is positive, contractive and sub-Markovian for $t>0$. Suppose that $F$ is a bounded holomorphic function in the sector
$\Sigma_{\pi/2}=\{z \in \mathbb C\setminus \{0\}: |\arg (z)| <\pi/2\}$. Then
\[
\norm{F(L) f}_p \leq C \norm{f}_p,\,\,\forall f \in L^p(M),
\]
where $C$ depends on $p$, $\sigma$ and $F$.
\end{thm}

\noindent{\bf Proof of Theorem 2.3:}
First assume (\ref{Rp}). For any $f \in C_0^{\infty}(M)$, on the one hand, we have
\begin{eqnarray*}
\norm{\abs{\nabla (I+\Delta)^{-1/2} f}}_p &=& \norm{\abs{\nabla \Delta^{-1/2} \Delta^{1/2} (I+\Delta)^{-1/2}f}}_p
\\ &\leq& C \norm{\Delta^{1/2} (I+\Delta)^{-1/2}f}_p 
\leq C \Vert f\Vert_p.
\end{eqnarray*}
Here the last inequality follows from Theorem \ref{fc}.

On the other hand, (\ref{Rp-infty}) holds obviously due to the $L^p$ boundedness of the heat semigroup. In fact,
\[
\norm{\abs{\nabla e^{-\Delta}\Delta^{-1/2} f}}_p \leq C \norm{e^{-\Delta}f}_p \leq C \Vert f\Vert_p.
\]

Conversely, assume (\ref{Rp-loc}) and (\ref{Rp-infty}), then
\begin{eqnarray*}
\norm{\abs{\nabla f}}_p 
&\leq& 
\norm{\abs{\nabla e^{-\Delta} f}}_p +\norm{\abs{\nabla (I-e^{-\Delta}) f}}_p
\\ &\lesssim&
\norm{\Delta^{1/2} f}_p +\norm{\abs{\nabla (I+\Delta)^{-1/2} (I+\Delta)^{1/2} (I-e^{-\Delta}) \Delta^{-1/2} \Delta^{1/2} f}}_p
\\ &\lesssim&
\norm{\Delta^{1/2} f}_p +\norm{(I+\Delta)^{1/2} (I-e^{-\Delta}) \Delta^{-1/2} \Delta^{1/2} f}_p
\\ &\lesssim& 
\norm{\Delta^{1/2} f}_p.
\end{eqnarray*}
Here the last inequality is due to Theorem \ref{fc}.
\hfill{$\Box$}

\bigskip

We shall now introduce a variation of the Riesz transform at infinity. Let $0<\alpha<1/2$. We say that $M$ satisfies 
($R_p^{\infty,\alpha}$) if
\begin{equation} \label{Rp-inftya} \tag{$R_p^{\infty,\alpha}$}
\norm{\abs{\nabla e^{-\Delta} f}}_p \leq C\norm{\Delta^{\alpha} f}_p, \,\,\forall f \in \mathcal C_0^{\infty}(M).
\end{equation}
Together with the local Riesz transform, it will give us a notion of quasi Riesz transform.

Note that (\ref{Rp-infty}) implies (\ref{Rp-inftya}). Indeed,
\begin{eqnarray*}
\norm{\abs{\nabla e^{-\Delta} \Delta^{-\alpha} f}}_p 
&\leq& 
\norm{\abs{\nabla e^{-\Delta/2} \Delta^{-1/2} e^{-\Delta/2} \Delta^{1/2-\alpha} f}}_p 
\\ &\leq&
C \norm{e^{-\Delta/2} \Delta^{1/2-\alpha} f}_p 
\\ &\leq&
C\Vert f\Vert_p.
\end{eqnarray*}

\bigskip
%%%%%%%%%%%%%%%%%%%%%%%%%%%%%%%%%%%%%%%%%%%%%%%%%%%%%%%%%
\subsection{Equivalence of (\ref{Gp}) and (\ref{MI})}

For any $1< p<\infty$, let us considering the following $L^p$ interpolation or multiplicative inequality:
\begin{equation}\label{MI}\tag{$M\!I_p$}
\norm{\abs{\nabla f}}_p^2 \leq C \norm{\Delta f}_p \norm{f}_p,\,\,
\forall f\in \mathcal C_{0}^{\infty}(M),
\end{equation}
as well as the following $L^p$ estimate for the gradient of the heat semigroup:
\begin{equation} \label{Gp} \tag{$G_p$}
\norm{\abs{\nabla e^{-t\Delta }}}_{p\to p} \leq \frac{C_p}{\sqrt t}, \,\,\forall t>0.
\end{equation}

Recall that (\ref{Rp}) implies (\ref{Gp}) and (\ref{MI}). In fact, (\ref{Gp}) and (\ref{MI}) are equivalent for any $1< p<\infty$.
\begin{prop}[\cite{CS10,Du08}] 
\label{GpMI}
Let $M$ be a complete Riemannian manifold. Then, for any $1< p<\infty$, (\ref{Gp}) is equivalent to (\ref{MI}).
\end{prop}  
See \cite{CS10} for more information about the relations between (\ref{MI}), the Riesz 
transforms, and estimates of the derivative of the heat kernel. For the sake of completeness, we give a proof here.

\begin{proof} 
First assume (\ref{MI}). Substituting $f$ by $e^{-t\Delta}f$ in (\ref{MI}) yields
\[
\norm{\abs{\nabla e^{-t\Delta}f}}_{p}^2 \leq C \norm{\Delta e^{-t\Delta}f}_p \norm{e^{-t\Delta}f}_p.
\]
Since the heat semigroup is analytic on $L^p(M)$, we obtain
\[
\norm{\abs{\nabla e^{-t\Delta}f}}_{p} \leq C t^{-1/2}\Vert f\Vert_p.
\]

Conversely assume (\ref{Gp}). For any $f\in C_0^{\infty}(M)$, write the identity
\[
f=e^{-t\Delta}f+\int_{0}^{t}\Delta e^{-s\Delta}f ds,\,\forall t>0.
\]
Then (\ref{Gp}) yields
\begin{eqnarray*}
\norm{\abs{\nabla f}}_p &\leq& 
C\norm{\abs{ \nabla e^{-t\Delta}f}}_p+\left\Vert\int_{0}^{t} \abs{\nabla \Delta e^{-s\Delta}f} ds\right\Vert_p
\\ &\leq& 
C t^{-1/2}\Vert f\Vert_p+\int_{0}^{t} \norm{\abs{\nabla e^{-s\Delta}\Delta f}}_p ds
\\ &\leq& 
C t^{-1/2}\Vert f\Vert_p+C t^{1/2} \Vert \Delta f\Vert_p.
\end{eqnarray*}
Taking $t=\Vert f\Vert_p \Vert \Delta f\Vert_p^{-1}$, we get (\ref{MI}).
\end{proof}

Under (\ref{Gp}) (or equivalently (\ref{MI})), the quasi Riesz transform at infinity is $L^p$ bounded: 
\begin{prop}\label{eq}
Let $M$ be a complete Riemannian manifold satisfying (\ref{Gp}) for some $p\in (1,\infty)$. Then for any $\alpha \in (0,1/2)$, 
(\ref{Rp-inftya}) holds.
\end{prop}

\begin{proof}
For any $f\in \mathcal C_{0}^{\infty}(M)$, write
\[
\nabla e^{-\Delta}\Delta^{-\alpha}f=\int_0^{\infty}\nabla e^{-(1+t)\Delta}f \frac{dt}{t^{1-\alpha}}.
\]
Since (\ref{Gp}) holds, we have
\[
\norm{\abs{\nabla e^{-\Delta}\Delta^{-\alpha}f}}_p 
\leq 
\int_0^{\infty}\norm{\abs{\nabla e^{-(1+t)\Delta}f}}_p \frac{dt}{t^{1-\alpha}}
\leq 
C_p \Vert f\Vert_p \int_0^{\infty} \frac{dt}{(t+1)^{1/2}t^{1-\alpha}},
\]
which obviously converges for $\alpha \in (0,1/2)$. Therefore we obtain (\ref{Rp-inftya}) for $\alpha \in (0,1/2)$.
\end{proof}

\bigskip

%%%%%%%%%%%%%%%%%%%%%%%%%%%%%%%%%%%%%%%%%%%%%%%%%%%%%%%%%
\subsection{$L^p$ boundedness of quasi Riesz transform for $1<p\leq2$}
This part is inspired by \cite{CD03} and \cite{Du08}, where (\ref{MI}) and (\ref{Gp}) for $1<p\leq 2$ were 
shown on manifolds and graphs respectively.

In the following, we will give a different proof of (\ref{MI}) and (\ref{Gp}) on Riemannian manifolds. More precisely, we will directly show 
(\ref{MI}) and (\ref{Gp}) by the method which is used in \cite{St70} to prove the $L^p$
boundedness of the Littlewood-Paley-Stein function (see also \cite[Theorem 1.2]{CDL03}). In \cite[Theorem 1.3]{Du08},  
an analogue proof was given in the discrete case.

\begin{prop}\label{GP}
Let $M$ be a complete Riemannian manifold. Then (\ref{MI}) and (\ref{Gp}) hold for $1<p\leq 2$.
\end{prop}

\begin{proof}
Assume that $f\in C_{0}^{\infty}(M)$ is non-negative and not identically zero. Set 
$u(x,t)=e^{-t\Delta}f(x)$. Then $u$ is smooth and positive everywhere. Moreover, $u(\cdot,t)$ and $\Delta u(\cdot,t)$ is $L^p$ bounded for 
$1<p<\infty$.

For any $1<p\leq 2$, we have
\begin{eqnarray*}
&& \left( \frac{\partial}{\partial t}+\Delta \right) u^p(x,t) 
\\&=& 
p u^{p-1}(x,t) \left(\frac{\partial}{\partial t} +\Delta\right) u(x,t) - p(p-1)u^{p-2}(x,t)|\nabla u(x,t)|^{2}
\\&=&
- p(p-1)u^{p-2}(x,t)|\nabla u(x,t)|^{2}.
\end{eqnarray*}
Define $J(x,t)= -\left( \frac{\partial}{\partial t}+\Delta \right) u^p(x,t)$, then
\begin{equation}\label{eqj}
|\nabla u(x,t)|^{2}=\frac{1}{p(p-1)}u^{2-p}(x,t)J(x,t).
\end{equation}

We first construct a sequence of functions $\{\phi_n\}$ as follows (see for example \cite{Da90}):
let $\eta(t), 0\leq t<\infty$ be a non-increasing smooth function such that $\eta(t)=1$ for $0<t<1$ and $\eta(t)=0$ for
$2\leq t<\infty$. Define $$\phi_n(x)=\eta(d(x,x_0)/n),$$ for any fixed point $x_0\in M$. 
Then 
\begin{itemize}
\item $0\leq \phi_n\leq 1$ and $\phi_n$ is a sequence of continuous functions which converges monotonically to $1$.
\item $\phi_n\in C_0^1(M)$ and 
\[
|\nabla \phi_n(x)| \leq \frac{\norm{\eta'}_{\infty}}{n},\,\, \forall x\in M.
\]
\end{itemize}

Then it follows from the Green's formula that
\begin{equation*}
\begin{split}
\int_M J(x,t) \phi_n^2(x) d\mu(x) 
&= -\int_M \frac{\partial}{\partial t} u^p(x,t) \phi_n^2(x) d\mu(x) 
-\int_M \Delta u^p(x,t)  \phi_n^2(x) d\mu(x)
\\ &=-\int_M \frac{\partial}{\partial t} u^p(x,t) \phi_n^2(x) d\mu(x) 
-\int_M \nabla u^p(x,t) \cdot \nabla \phi_n^2(x) d\mu(x).
\end{split}
\end{equation*}
%The left hand side converges to  $\int_M J(x,t) d\mu(x)$ form Lebesgue's monotone convergence theorem, 
%since $J(x,t)$ is non-negative and $\phi_n^2$ converges monotonically to $1$.
Let us estimate the right hand side. The first integral converges to $\int_M \frac{\partial}{\partial t} u^p(x,t) d\mu(x)$, which is integrable. Indeed, we get from the H\"older inequality that
\begin{equation}\label{der}
\begin{split}
\int_M \left|\frac{\partial}{\partial t} u^p(x,t)\right| d\mu(x)
&=\int_M p u^{p-1}(x,t) |\Delta u(x,t)| d\mu(x)
\\ &\leq  
C \norm{u(\cdot,t)}_p^{p-1} \norm{\Delta u(\cdot,t)}_p.
\end{split}
\end{equation}
Now we move to the second integral. Note that
\[
\int_M \nabla u^p(x,t) \cdot \nabla \phi_n^2(x) d\mu(x) 
= \int_M 2pu^{p-1} \phi_n(x)\nabla u(x,t) \cdot \nabla \phi_n(x) d\mu(x).
\]
Still from the H\"older inequality, it holds
\begin{equation}\label{nabla}
\begin{split}
\abs{\int_M \nabla u^p(x,t) \cdot \nabla \phi_n^2(x) d\mu(x)}
&\leq 2p\norm{|\nabla \phi_n|}_{\infty} \int_M u^{p-1}(x,t)(\phi_n |\nabla u(x,t)|) d\mu(x)
\\ & \leq \frac{C}{n} \norm{\phi_n |\nabla u(\cdot,t)|}_{p} \norm{u}_p^{p-1}.
\end{split}
\end{equation}
By using the above two estimates (\ref{der}) and (\ref{nabla}), we get
\begin{equation}\label{J1}
\int_{M} J(x,t)\phi_n^2(x) d\mu(x) \leq
C \norm{u(\cdot,t)}_p^{p-1} \norm{\Delta u(\cdot,t)}_p+\frac{C}{n} \norm{\phi_n |\nabla u(\cdot,t)|}_{p} \norm{u}_p^{p-1}.
\end{equation}

Note also that from (\ref{eqj}), (\ref{J1}) and the H\"older inequality, 
\begin{equation*}
\begin{split}
&\norm{\phi_n|\nabla u(\cdot,t)|}_{p}^p
= C\int_M (u^{2-p}(x,t)J(x,t)\phi_n^2(x))^{p/2} d\mu
\\ &\leq C \norm{u}_p^{p(2-p)/2} \br{\int_{M} J(x,t)\phi_n^2(x)d\mu(x)}^{p/2}
\\ &\lesssim  
\norm{u}_p^{p(2-p)/2} \br{\norm{u(\cdot,t)}_p^{p-1} \norm{\Delta u(\cdot,t)}_p+\frac{1}{n} 
\norm{\phi_n |\nabla u(\cdot,t)|}_{p} \norm{u}_p^{p-1}}^{p/2}
\\ &\lesssim  
\norm{u}_p^{p(2-p)/2} \br{\norm{u(\cdot,t)}_p^{p(p-1)/2} \norm{\Delta u(\cdot,t)}_p^{p/2}+\frac{1}{n^{p/2}} 
\norm{\phi_n |\nabla u(\cdot,t)|}_{p}^{p/2} \norm{u}_p^{p(p-1)/2}}
\\ &\lesssim 
\norm{u(\cdot,t)}_p^{p/2} \br{\norm{\Delta u(\cdot,t)}_p^{p/2}+\frac{1}{n^{p/2}} \norm{\phi_n |\nabla u(\cdot,t)|}_{p}^{p/2}}.
\end{split}
\end{equation*}
Therefore, we have
\begin{eqnarray*}
\norm{\phi_n |\nabla u(\cdot,t)|}_{p}^p
\lesssim \frac{1}{n^p}\norm{u(\cdot,t)}_p^{p} +\norm{u(\cdot,t)}_p^{p/2}\norm{\Delta u(\cdot,t)}_p^{p/2}.
\end{eqnarray*}
As $n$ goes to infinity, the left hand side converges to $\int_M |\nabla u|^p d\mu$ from Lebesgue's monotone convergence theorem. Finally we obtain
\begin{equation}\label{hd}
\norm{|\nabla u(\cdot,t)|}_{p}^p
\lesssim \norm{u(\cdot,t)}_p^{p/2}\norm{\Delta u(\cdot,t)}_p^{p/2}.
\end{equation}

On the one hand, as $t$ goes to zero, we get the multiplicative inequality from (\ref{hd}) that 
\[
\norm{\abs{\nabla f}}_p^p \leq C \Vert f\Vert_p^{p/2} \Vert \Delta f \Vert_p^{p/2}.
\]
 
On the other hand, by the analyticity of the heat semigroup, (\ref{hd}) yields
\[
\norm{\abs{\nabla u(\cdot, t)}}_p^p \leq C t^{-p/2} \Vert f \Vert_p^{p},
\]
which is exactly (\ref{Gp}).
\end{proof}

\begin{rem}
Note that Proposition \ref{GP} can not be extended to the case $p>2$ without additional assumptions. Indeed (\ref{Gp}) for $p>2$ 
has consequences that are not always true, see \cite{ACDH04}.
\end{rem}

Combining Proposition \ref{eq} and Proposition \ref{GP}, we get
\begin{cor}
Let $M$ be a complete Riemannian manifold. Then for any fixed $\alpha \in (0,1/2)$, the operator 
$\nabla e^{-\Delta}\Delta^{-\alpha}$ is $L^p$ bounded for $1<p\leq 2$.
\end{cor}

\bigskip
%%%%%%%%%%%%%%%%%%%%%%%%%%%%%%%%

\section{Sub-Gaussian heat kernel estimates and quasi Riesz transforms}

Remember that with local assumptions on the manifold, we get the $L^p$ ($1<p\leq 2$) boundedness of quasi Riesz transforms 
$\nabla(I+\Delta)^{-1/2}+\nabla e^{-\Delta }\Delta^{-\alpha}$, where $0<\alpha<1/2$. If we assume in addition a global Gaussian 
heat kernel upper bound, the Riesz transform itself is $L^p$ bounded for $1<p\leq 2$ and weak $(1,1)$ bounded. What  
happens if we suppose globally another heat kernel upper bound, the so-called sub-Gaussian upper bound  (\ref{ue})?

In the case of Riemannian manifolds satisfying (\ref{doubling}) and (\ref{ue}), Proposition \ref{exo} tells us that the quasi Riesz 
transform is $L^p$ bounded for $1<p\leq 2$. Yet we don't know whether the Riesz transform, which corresponds to $\alpha=1/2$, 
is $L^p$ bounded or not for $1<p\leq 2$. Instead, we will study the endpoint case for the quasi Riesz transform, that is, what 
happens for $p=1$?  In the following, we will prove the weak $(1,1)$ boundedness of the quasi Riesz transform. 
%%%%%%%%%%%%%%%%%%%%%%%%%%%%%%%%%%%%%%%%%%%%%%%%%%%%%%%%%

\bigskip

\subsection{More about sub-Gaussian heat kernel estimate}
One can rewrite (\ref{ue}) as follows:
\[
p_{t}(x,y) \leq\left\{ \begin{aligned}
& \frac{C}{V(x,t^{1/2})}\exp\br{-c \frac{d^2(x,y)}{t}}, &t<\min\{1,d(x,y)\},\\
& \frac{C}{V(x,t^{1/2})}\exp\br{-c \br{\frac{d^m(x,y)}{t}}^{1/(m-1)}}, &d(x,y)\leq t<1,\\
& \frac{C}{V(x,t^{1/m})}\exp\br{-c \frac{d^2(x,y)}{t}}, &1\leq t<d(x,y),\\
& \frac{C}{V(x,t^{1/m})}\exp\br{-c\br{\frac{d^m(x,y)}{t}}^{1/(m-1)}}, &t\geq \max \{1, d(x,y)\}.
\end{aligned}\right.
\]

Note that for $d(x,y)\leq t$, one has $\frac{d^2(x,y)}{t} \leq \br{\frac{d^m(x,y)}{t}}^{1/(m-1)}$. And for $t \leq d(x,y)$, one has 
$\frac{d^2(x,y)}{t} \geq \br{\frac{d^m(x,y)}{t}}^{1/(m-1)}$. Thus we have the following estimate:
\begin{equation}\label{ls}
p_{t}(x,y) \leq\left\{ \begin{aligned}
& \frac{C}{V(x,t^{1/2})}\exp\br{-c \frac{d^2(x,y)}{t}}, &0<t<1,\\
& \frac{C}{V(x,t^{1/m})}\exp\br{-c \br{\frac{d^m(x,y)}{t}}^{1/(m-1)}}, &t\geq 1.
\end{aligned}\right.
\end{equation}

That is, the small time behaviour of the heat kernel is Gaussian as in Euclidean spaces while the heat kernel has a sub-Gaussian 
decay for large time. 

There exist such manifolds for all $m \geq 2$. One can choose any $D\geq 1$ and any $2\leq m\leq D+1$ such that 
$V(x,r) \simeq r^D$ for $r\geq 1$ and (\ref{ue}) holds. Indeed, fractal manifolds, which are built from graphs with a fractal  
structure at infinity, provide examples satisfying (\ref{ue}) with some $m>2$ (in fact, two-sided sub-Gaussian heat kernel estimates). 
We refer to \cite{Ba04} for the construction of suitable graphs. For a concrete example, Barlow, Coulhon and Grigor'yan in 
\cite{BCG01} constructed such a manifold whose discretisation is the Vicsek graph. For more examples, see the work of Barlow 
and Bass \cite{BB92}, \cite{BB99a}, \cite{BB99b}. We also refer to \cite{GT12,HSC01} for more general non-classical heat kernel 
estimates on metric measure spaces.
\medskip

\emph{Comparison with the Gaussian heat kernel estimate (\ref{UE}):}
\[
p_t(x,y)\leq \frac{C}{V(x,\sqrt t)} \exp\left(-c\frac{d^2(x,y)}{t}\right),\,\,\forall x,y\in M, t>0.
\]
Since $m>2$, $p_t(x,x)$ decays with $t$ more slowly in the sub-Gaussian case than in the Gaussian case. Also for 
$t\geq \max\{1,d(x,y)\}$, $p_t(x,y)$ decays with $d(x,y)$ faster in the sub-Gaussian case than in the Gaussian case. Therefore 
the two kinds of pointwise estimates are not comparable. 

\bigskip
%%%%%%%%%%%%%%%%%%%%%%%%%%%%%%%%%%%%%%%%%%%%%%%%%%%%%

\subsection{Weighted estimates of the heat kernel}
Let $(M,d,\mu)$ be a non-compact complete manifold satisfying the doubling volume property (\ref{doubling}) and the 
sub-Gaussian estimate (\ref{ue}). In the following, we aim to get the integral estimates for the heat kernel and its time and space
derivatives. The method we use here is similar as in \cite[Section 2.3]{CD99}.  

First, we have the pointwise estimate of the time derivative of heat kernel:
\begin{lem}
\label{lem1} Let $M$ be as above, then we have
\begin{equation}\label{tdhk}
\left|\frac{\partial}{\partial t}p_t(x,y)\right|
\leq\left\{ 
\begin{aligned}
&\frac{C}{t V\left(y,t^{1/2}\right)}\exp\left(-c\frac{d^2(x,y)}{t}\right), &t< 1,\\
& \frac{C}{t V\left(y,t^{1/m}\right)}\exp\left(-c\left(\frac{d^m(x,y)}{t}\right)^{1/(m-1)}\right), &t\geq 1.\end{aligned} \right.
\end{equation}
\end{lem}

\begin{proof}
We see from \cite{Da97} (Thm. 4 and Cor. 5) that there exist an $a\in(0,1)$ such that for $0<t<a$
\[
\left|\frac{\partial}{\partial t}p_t(x,y)\right|\leq
\frac{C}{t V\left(y,t^{1/2}\right)}\exp\left(-c\frac{d^2(x,y)}{t}\right);
\]
and for $t>a^{-1}$,
\[
\left|\frac{\partial}{\partial t}p_t(x,y)\right|\leq
\frac{C}{t V\left(y,t^{1/m}\right)}\exp\left(-c\left(\frac{d^m(x,y)}{t}\right)^{1/(m-1)}\right).
\]

For $t\in (a,1)$, according to \cite{Da97}, Cor.5, it suffices to show that there exists a constant 
$\delta \in (0,1)$ such that for all $s\in [(1-\delta )t,(1+\delta )t]$
\[
p_s(x,y)\leq \frac{C}{V\left(y,s^{1/2}\right)}
\exp\left(-c\frac{d^2(x,y)}{s}\right).
\]
This is obvious since $V(x,s^{1/2})\simeq V(x,s^{1/m})$ and $\frac{d^2(x,y)}{t} \leq \br{\frac{d(x,y)}{t}}^{m/(m-1)}$ for $t\geq d(x,y)$.

The case for $t\in [1,a^{-1})$ is similar. Due to the facts $V(x,s^{1/2})\simeq V(x,s^{1/m})$ and 
$\frac{d^2(x,y)}{t} \geq \br{\frac{d(x,y)}{t}}^{m/(m-1)}$ for $t\leq d(x,y)$, there exists a constant $\delta \in (0,1)$ such that for all 
$s\in [(1-\delta )t,(1+\delta )t]$,
\[
p_s(x,y)\leq \frac{C}{V\left(y,s^{1/m}\right)}\exp\left(-c\left(\frac{d^m(x,y)}{s}\right)^{1/(m-1)}\right).
\]

Therefore, we obtain (\ref{tdhk}).
\end{proof}

Now we intend to  estimate $\int_{B(x,r)^{c}}|\nabla p_t(x,y)|d\mu(x)$ for any $t>0$ and $r\geq 0$.

\begin{lem}
\label{lem2} For any $\alpha\in (1/m,1/2)$, we have for any $y\in M$ and $r\geq 0$,
\begin{equation}\label{we}
\int_{M\setminus B(y,r)}|\nabla p_t(x,y)|d\mu(x)\leq\left\{
\begin{aligned}
&C t^{-\frac{1}{2}}e^{-c\frac{r^2}{t}},\,\,\,& 0<t<1, \\
&C t^{-\alpha}e^{-c\left(\frac{r^m}{t}\right)^{\frac{1}{m-1}}},\,\,\,& t\geq 1.
\end{aligned}\right.
\end{equation}
\end{lem}

\begin{rem}
Note that the estimate (\ref{we}) holds for any $0<\alpha<1/2$. But in the proof below, $\alpha$ can not achieve $1/2$ unless 
$m=2$. This allows us to obtain the weak $(1,1)$ boundedness of $\nabla e^{-\Delta} \Delta^{-\alpha}$, not the Riesz transform. 
If one could get (\ref{we}) with $\alpha=1/2$, the proof in Section 4 below would yield the boundedness of the Riesz transform.
\end{rem}

\begin{proof}
For $0<t<1$, the above estimate is proved in \cite{CD99}.

\medskip

Now for $t\geq 1$. Comparing with the proof of the Gaussian case in \cite{CD99}, we need to replace the weight 
$\exp\left(-c\frac{d^2(x,y)}{t}\right)$ by $\exp\left(-c\left(\frac{d^m(x,y)}{t}\right)^{1/(m-1)}\right)$ ($c$ is chosen 
appropriately).

\medskip

\noindent \underline{Step 1}: For any $c>0$,
\begin{equation}
\label{est}
\int_{M\setminus B(y,r)}\exp\left(-c\left(\frac{d^m(x,y)}{t}\right)^{1/(m-1)}\right)d\mu(x)
\lesssim e^{-\frac{c}{2}\left(\frac{r^m}{t}\right)^{1/(m-1)}}V\left(y,t^{1/m}\right).
\end{equation}
Indeed,
\begin{eqnarray*}
&&\int_{M\setminus B(y,r)}\exp\left(-c\left(\frac{d^m(x,y)}{t}\right)^{1/(m-1)}\right)d\mu(x)
\\ & \leq &
e^{-\frac{c}{2}\left(\frac{r^m}{t}\right)^{1/(m-1)}} \int_M \exp\left(-\frac{c}{2}\left(\frac{d^m(x,y)}{t}\right)^{1/(m-1)}\right)d\mu(x)
\\& \leq &
e^{-\frac{c}{2}\left(\frac{r^m}{t}\right)^{1/(m-1)}} \sum_{i=0}^{\infty}\int_{B(y,(i+1)t^{1/m})\setminus B(y,it^{1/m})}
\exp\left(-\frac{c}{2}\left(\frac{d^m(x,y)}{t}\right)^{1/(m-1)}\right)d\mu(x)
\\& \leq &
e^{-\frac{c}{2}\left(\frac{r^m}{t}\right)^{1/(m-1)}}V\left(y,t^{1/m}\right) \sum_{i=0}^{\infty}(i+1)^\nu e^{-\frac{c}{2}i^{1/(m-1)}}
\\&\leq &
C e^{-\frac{c}{2}\left(\frac{r^m}{t}\right)^{1/(m-1)}}V\left(y,t^{1/m}\right).
\end{eqnarray*}

\noindent \underline{Step 2}: For $0<\gamma<2c $ ($c$ is the constant in (\ref{ue})), we have
\begin{eqnarray*}
&& \int_M p_t(x,y)^2\exp\left(\gamma\left(\frac{d^m(x,y)}{t}\right)^{1/(m-1)}\right)d\mu(x)
\\ &\leq& 
\frac{C}{ V^2\left(y,t^{1/m}\right)} \int_M \exp\left((\gamma-2c)\left(\frac{d^m(x,y)}{t}\right)^{1/(m-1)}\right)d\mu(x)
\leq \frac{C_\gamma}{ V\left(y,t^{1/m}\right)}.
\end{eqnarray*}
This is a consequence of (\ref{ue}) and Step 1 with $r=0$.

\medskip

\noindent \underline{Step 3}: Denote 
\[
I(t,y)=\int_M |\nabla_x p_t(x,y)|^2
\exp\left(\gamma\left(\frac{d^m(x,y)}{t}\right)^{1/(m-1)}\right)d\mu(x),
\]
with $\gamma$ small enough. Using integration by parts,
\begin{eqnarray*}
I(t,y) & = & \int_M p_t(x,y)\Delta p_t(x,y)\exp\left(\gamma\left(\frac{d^m(x,y)}{t}\right)^{1/(m-1)}\right)d\mu(x)
\\&&
-\int_M p_t(x,y)\nabla_x p_t(x,y)\cdot\nabla_x\exp\left(\gamma\left(\frac{d^m(x,y)}{t}\right)^{1/(m-1)}\right)d\mu(x)
\\& = &
-\int_M p_t(x,y)\frac{\partial}{\partial t} p_t(x,y)\exp\left(\gamma\left(\frac{d^m(x,y)}{t}\right)^{1/(m-1)}\right)d\mu(x)
\\&&
-\frac{\gamma m}{m-1}\int_M p_t(x,y)\nabla_x p_t(x,y)\left(\frac{d(x,y)}{t}\right)^{1/(m-1)}
\\&&
\cdot\nabla_x d(x,y)\exp\left(\gamma\left(\frac{d^m(x,y)}{t}\right)^{1/(m-1)}\right)d\mu(x)
\\&=& I_1(t,y)+I_2(t,y).
\end{eqnarray*}

According to Lemma \ref{lem1} and Step 1,
\[
|I_1(t,y)|\leq \frac{ C_\gamma'}{t V\left(y,t^{1/m}\right)}.
\]

For $I_2$, since $|\nabla_x d(x,y)|\leq 1$ and $\left(\frac{d(x,y)}{t}\right)^{1/(m-1)}=\left(\frac{d^m(x,y)}{t}\right)^{1/m(m-1)}t^{-1/m}$,
then from Step 2 and Cauchy-Schwartz inequality,
\[
|I_2(t,y)| \leq C_{\gamma}'' t^{-1/m}(I(t,y))^{1/2}\left(\frac{C_\gamma}{V\left(y,t^{1/m}\right)}\right)^{1/2}.
\]

We get
\begin{eqnarray*}
I(t,y) &\leq & 
\frac{C_\gamma'}{t V\left(y,t^{1/m}\right)}+t^{-1/m}(I(t,y))^{\frac{1}{2}}\left(\frac{C_\gamma''}{V\left(y,t^{1/m}\right)}\right)^{1/2}
\\&\leq &
\frac{C_\gamma'}{t^{2/m}V\left(y,t^{1/m}\right)}+(I(t,y))^{1/2}\left(\frac{C_\gamma''}{t^{2/m}V\left(y,t^{1/m}\right)}\right)^{1/2}.
\end{eqnarray*}

Therefore
\begin{equation}\label{wd}
I(t,y)\leq \frac{C}{t^{2/m} V\left(y,t^{1/m}\right)}.
\end{equation}

\noindent \underline{Step 4}: We divide the integral $\int_{M\setminus B(y,r)} \abs{\nabla p_t(x,y)} d\mu(x)$ as follows
\begin{eqnarray*}
\int_{M\setminus B(y,r)} \abs{\nabla p_t(x,y)} d\mu(x)
&=& 
\sum_{i=0}^{\infty} \int_{2^i r< d(x,y)\leq 2^{i+1}r} \abs{\nabla p_t(x,y)} d\mu(x)
\\ &\leq& 
\sum_{i=0}^{\infty} V^{1/2}\br{y, 2^{i+1}r} \br{\int_{2^i r< d(x,y)\leq 2^{i+1}r} \abs{\nabla p_t(x,y)}^2 d\mu(x)}^{1/2}.
\end{eqnarray*}

For each $i \geq 0$, it follows from (\ref{wd}) that
\begin{equation}\label{sq1}
\begin{split}
&\br{\int_{2^i r< d(x,y)\leq 2^{i+1}r} \abs{\nabla p_t(x,y)}^2 d\mu(x)}^{1/2}
\\ \leq &
\br{\int_{2^i r< d(x,y)\leq 2^{i+1}r} \abs{\nabla p_t(x,y)}^2 \exp\br{c \br{\frac{d^m(x,y)}{t}}^{1/(m-1)}} d\mu(x)}^{1/2}
\cdot e^{-c \br{\frac{2^{im}r^m}{t}}^{1/(m-1)}}
\\ \leq &
\frac{C}{t^{1/m}V^{1/2}(y,t^{1/m})} e^{-c \br{\frac{2^{im}r^m}{t}}^{1/(m-1)}}.
\end{split}
\end{equation}

On the other hand, applying (\ref{est}) with $r=0$ (as well as the corresponding estimate for $t/2<1$),
\begin{equation}\label{sq2}
\begin{split}
\br{\int_{2^i r< d(x,y)\leq 2^{i+1}r} \abs{\nabla p_t(x,y)}^2 d\mu(x)}^{1/2}
 \leq &
\norm{\abs{\nabla e^{-\frac{t}{2}\Delta}}}_{2\rightarrow 2} \norm{p_{\frac{t}{2}}(\cdot,y)}_{2}
\\ \leq& 
C t^{-1/2}  p_t^{1/2}(y,y)
\\ \leq &
\frac{C}{t^{1/2} V^{1/2}\br{y, t^{1/m}}}.
\end{split}
\end{equation}
The second inequality follows from the fact $\norm{p_{\frac{t}{2}}(\cdot,y)}_{2}^2=p_t(y,y)$.

Thus taking $\theta=\frac{\frac{1}{2}-\alpha}{\frac{1}{2}-\frac{1}{m}}$, we get from (\ref{sq1}) and (\ref{sq2}) that
\begin{equation}\label{l2}
\begin{split}
&\br{\int_{2^i r< d(x,y)\leq 2^{i+1}r} \abs{\nabla p_t(x,y)}^2 d\mu(x)}^{1/2}
\\ \leq& 
\br{\frac{C}{t^{1/m}V^{1/2}(y,t^{1/m})} e^{-c \br{\frac{2^{im}r^m}{t}}^{1/(m-1)}}}^{\theta} 
\br{\frac{C}{t^{1/2} V^{1/2}\br{y, t^{1/m}}}}^{1-\theta}
\\ \leq& 
\frac{C}{t^{\alpha} V^{1/2}\br{y, t^{1/m}}} e^{-c \br{\frac{2^{im}r^m}{t}}^{1/(m-1)}},
\end{split}
\end{equation}
where $c$ depends on $\alpha$.

Finally (\ref{D1}) and (\ref{l2}) yield
\begin{eqnarray*}
\int_{M\setminus B(y,r)} \abs{\nabla p_t(x,y)} d\mu(x)
&\leq& 
\sum_{i=0}^{\infty} V^{1/2}\br{y, 2^{i+1}r} \frac{C}{t^{\alpha} V^{1/2}\br{y, t^{1/m}}} e^{-c \br{\frac{2^{im}r^m}{t}}^{1/(m-1)}}
\\ &\leq& 
C t^{-\alpha} e^{-c\br{\frac{r^m}{t}}^{1/(m-1)}}.
\end{eqnarray*}
\end{proof}

\bigskip
%%%%%%%%%%%%%%%%%%%%%%%%%%%%%%%%%%%%%%%%%%%%%%%%%%%%%%%%%%%%%%%

\subsection{Weak $(1,1)$ boundedness of quasi Riesz transforms }
In order to show the weak $(1,1)$ boundedness of quasi Riesz transform, we will use the Calder\'on-Zygmund decomposition. 
Let us recall the result:
\begin{thm}\label{C-Z}
Let $(M,d,\mu)$ be a metric measured space satisfying the doubling volume property. Then for any given function $f\in L^1(M)\cap 
L^2(M)$ and $\lambda >0$, there exists a decomposition of $f$, $f=g+b=g+\sum_i b_i $ so that

\begin{enumerate}
\item $|g(x)|\leq C\lambda$  for almost all $x\in M$;

\item There exists a sequence of balls $B_i =B(x_i,r_i)$ so that each $b_i$ is supported in $B_i$,
\[
\int| b_i(x)|d\mu(x)\leq C\lambda\mu(B_i)\text{ and } 
\int b_i(x)d\mu(x)=0; 
\]

\item $\sum_i \mu(B_i)\leq \frac{C}{\lambda}\int|f(x)|d\mu(x)$;

\item $\Vert b\Vert_1\leq C\Vert f\Vert_1$ and $\Vert g\Vert_1\leq(1+C)\Vert f\Vert_1$;

\item There exists $k\in \mathbb{N}^*$ such that each $x\in M$ is contained in at most $k$ balls $B_i$.
\end{enumerate}
\end{thm}

We refer to \cite{CW71} and \cite{St701} for the proof. 
\medskip

Our result is
\begin{thm}
Let M be a complete Riemannian manifold satisfying (\ref{doubling}) and (\ref{ue}). Then for any $0<\alpha<1/2$, the quasi Riesz 
transform $\nabla(I+\Delta)^{-1/2}+\nabla e^{-\Delta} \Delta^{-\alpha}$ is of weak type $(1,1)$.
\end{thm}

\begin{rem}
By Marcinkiewicz interpolation theorem, this gives back Theorem \ref{thm1}, but under much stronger assumptions.
\end{rem}

\begin{rem}
In the following proof, we will adopt the singular integral technique used by Coulhon and Duong in \cite{CD99}, which was first 
developed by Duong and McIntosh in \cite{DM99}. 
\end{rem}

\begin{proof}
Note that the local Riesz transform $\nabla(I+\Delta)^{-1/2}$ is of weak type $(1,1)$ ( see Theorem \ref{Rploc}).
Denote $T=\nabla e^{-\Delta} \Delta^{-\alpha}$, it remains to show that 
\[
\mu(\{x: |T f(x)|>\lambda\})\leq C\lambda^{-1}\Vert f\Vert_1.
\]

Fix $f\in L^1(M)\cap L^2(M)$, we take the Calder\'{o}n-Zygmund decomposition of $f$ at the level of $\lambda$, i. e., 
$f=g+b=g+\sum_ib_i$, then
\[
\mu(\{x: |T f(x)|>\lambda\})
\leq\mu(\{x:|T g(x)|> \lambda/ 2 \})+\mu(\{x:|T b(x)|>\lambda/ 2 \}).
\]

Since $T$ is $L^2$ bounded, by using Theorem \ref{C-Z} we get
\[
\mu(\{x: |T g(x)|>\lambda/ 2\}) 
\leq C\lambda^{-2}\Vert g\Vert_2^2 \leq C\lambda^{-1}\Vert g\Vert_1
\leq C\lambda^{-1}\Vert f\Vert_1.
\]

As for the second term, we divide $\{B_i\}$ into two classes: the one in which the balls have radius no less than $1$ and the one in 
which the balls have radius smaller than $1$. Denote by
\begin{eqnarray*}
&& \mathcal {C}_1=\{i:B_i=B(x_i,r_i) \text{ with } r_i\geq 1\};
\\&& \mathcal {C}_2=\{i:B_i=B(x_i,r_i) \text{ with } r_i< 1\}.
\end{eqnarray*}
Then we have
\begin{eqnarray*}
\mu(\{x: |T \sum_{i}b_i(x)|>\lambda/ 2\}) 
& \leq & \mu(\{x: |T \sum_{i\in \mathcal {C}_1}b_i(x)|>\lambda/ 4\})
\\ &&+ \mu(\{x: |T \sum_{i\in\mathcal {C}_2}b_i(x)|>\lambda/ 4\}).
\end{eqnarray*}

Write 
\begin{equation}\label{divide}
T b_i=T e^{-t_i\Delta }b_i+T \left(I-e^{-t_i\Delta }\right)b_i,
\end{equation}
where $t_i=\rho(r_i)$ with $\rho$ defined in (\ref{t}). In the following, we will consider the two cases of balls separately.
 
\medskip

\noindent{\bf Case 1: }For balls with radius no less than $1$, our aim here is to prove 
\[
\mu(\{x: |T \sum_{i\in \mathcal {C}_1}b_i(x)|>\lambda/ 4\}) \leq C \lambda^{-1}\Vert f\Vert_1.
\]

Using (\ref{divide}), we have
\begin{eqnarray*}
\mu(\{x: |T \sum_{i\in \mathcal {C}_1}b_i(x)|>\lambda/ 4\}) 
& \leq & \mu(\{x: |T \sum_{i\in \mathcal {C}_1}e^{-t_i\Delta }b_i(x)|>\lambda/ 8\})
\\ &&+ \mu(\{x: |T \sum_{i\in \mathcal {C}_1}(I-e^{-t_i\Delta })b_i(x)|>\lambda/ 8\}).
\end{eqnarray*}

We begin to estimate the first term. Since $T$ is $L^2$ bounded, then
\[
\mu(\{x: |T \sum_{i\in \mathcal {C}_1}e^{-t_i\Delta }b_i(x)|>\lambda/ 8\})
\leq \frac{C}{\lambda^2}\left\Vert \sum_{i\in \mathcal {C}_1}e^{-t_i\Delta }b_i\right\Vert_2^2.
\]

By a duality argument, 
\[
\norm{\sum_{i\in \mathcal {C}_1}e^{-t_i\Delta }b_i}_2 
= \sup_{\norm{\phi}_2=1} \abs{<\sum_{i\in\mathcal {C}_1}e^{-t_i\Delta }b_i,\phi>}.
\]

It holds that for $t\geq 1$ and $\phi\geq 0$
\[
\sup_{y\in B(x,t^{1/m})} e^{-t\Delta } \phi(y) \leq C \inf_{y\in B(x,t^{1/m})} \mathcal M(\phi)(y),
\]
where $\M$ denotes the Littlewood-Paley maximal operator: 
\[
\M f(x)=\sup_{B\ni x} \frac{1}{\mu(B)} \int_B |f(y)| d\mu(y).
\]
Indeed, for any $y\in B(x,t^{1/m})$, we have
\begin{eqnarray*}
e^{-t\Delta } \phi(y)
&=& \int_{M}p_{t}(y,z) \phi(z) d\mu(z)
=\sum_{j=1}^{\infty} \int_{C_j(B(x,t^{1/m}))}p_{t}(y,z) \phi(z) d\mu(z)
\\ &\leq& C\sum_{j=1}^{\infty} \frac{V(x,2^{j+1}t^{1/m})}{V(y,t^{1/m})} e^{-c2^{jm/(m-1)}} \frac{1}{V(x,2^{j+1}t^{1/m})}\int_{B(x,2^{j+1}t^{1/m})} \phi(z) d\mu(z)
\\ &\leq& \inf_{y\in B(x,t^{1/m})} \mathcal M(\phi)(y).
\end{eqnarray*}

Then
\begin{eqnarray*}
\abs{<\sum_{i\in\mathcal {C}_1}e^{-t_i\Delta }b_i,\phi>}
&=& \abs{\sum_{i\in\mathcal {C}_1} <b_i, e^{-t_i\Delta } \phi>}
\leq 
\sum_{i\in\mathcal {C}_1} \int_M |b_i| d\mu \sup_{B_i} e^{-t_i\Delta }|\phi|
\\ &\leq&
C \sum_{i\in\mathcal {C}_1} \int_M |b_i| d\mu \inf_{B_i} \mathcal M(|\phi|)
\leq  
C \lambda \sum_{i\in\mathcal {C}_1} \int_{B_i}  \br{\M\br{|\phi|^2}(y)}^{1/2} d\mu(y) 
\\ &\leq & 
C \lambda \int_M \sum_{i\in\mathcal {C}_1} \mathbbm 1_{B_i}(y) \br{\M\br{|\phi|^2}(y)}^{1/2} d\mu(y)
\\ &\leq & 
C \lambda \int_{\cup_{i\in\mathcal {C}_1}B_i}  \br{\M\br{|\phi|^2}(y)}^{1/2} d\mu(y)
\\ &\leq&
C \lambda \mu^{1/2}(\cup_{i\in\mathcal {C}_1}B_i) \leq C \lambda^{1/2}\norm{f}_1^{1/2}.
\end{eqnarray*}
The inequality in the fourth line is due to the finite overlapping of the Calder\'on-Zygmund decomposition. In the first inequality of the last line,
we use Kolmogorov's lemma and weak type $(1,1)$ of the Hardy-Littlewood maximal function, as in \cite{HM03}.  

Therefore, we obtain
\[
\mu(\{x: |T \sum_{i\in \mathcal {C}_1}e^{-t_i\Delta }b_i(x)|>\lambda/ 8\}) \leq C \lambda^{-1}\norm{f}_1.
\]

It remains to show 
$\mu(\{x: |T \sum_{i\in \mathcal {C}_1}(I-e^{-t_i\Delta })b_i(x)|>\lambda/ 8\})\leq C \lambda^{-1} \Vert f\Vert_1$. 
We have
\begin{eqnarray*}
&&\mu(\{x: |T \sum_{i\in \mathcal {C}_1}(I-e^{-t_i\Delta })b_i(x)|>\lambda /8\})
\\&\leq & 
\mu(\{x\in \bigcup_{i\in \mathcal {C}_1}2B_i: |T \sum_{i\in \mathcal {C}_1}(I-e^{-t_i\Delta })b_i(x)|>\lambda /8\})
\\&& 
+\mu(\{x\in M\setminus \bigcup_{i\in \mathcal {C}_1}2B_i: |T \sum_{i\in \mathcal {C}_1}(I-e^{-t_i\Delta })b_i(x)|>\lambda /8\})
\\&\leq & 
\sum_{i\in \mathcal {C}_1}\mu (2 B_i)+\frac{8}{\lambda} \sum_{i\in \mathcal {C}_1}
\int_{M\setminus 2 B_i}|T (I-e^{-t_i\Delta })b_i(x)|d\mu(x).
\end{eqnarray*}

We claim: $\forall t\geq 1$, $\forall b$ with support in $B$, then
\[
\int_{M\setminus 2 B}|T(I-e^{-t\Delta })b(x)|d\mu(x)\leq C \Vert b\Vert_1.
\]
Therefore, by using Theorem \ref{C-Z}, we obtain
\[
\mu(\{x: |T \sum_{i\in \mathcal {C}_1}(I-e^{-t_i\Delta })b_i(x)|>\lambda/ 8\})
\leq \sum_{i\in \mathcal {C}_1}\mu (2 B_i)+\frac{C}{\lambda} \sum_{i\in \mathcal {C}_1} \norm{b_i}_1
\leq C \lambda^{-1} \Vert f\Vert_1.
\]

Denote by $k_t(x,y)$ the kernel of the operator $T(I-e^{-t_i\Delta })$, then
\begin{eqnarray*}
\int_{M\setminus 2 B}|T(I-e^{-t\Delta })b(x)|d\mu(x)
&\leq & 
\int_{M\setminus 2 B}\int_{B}|k_t(x,y)||b(y)|d\mu(y)d\mu(x)
\\&\leq &
\int_M |b(y)|\int_{d(x,y)\geq t^{1/m}}|k_t(x,y)|d\mu(x)d\mu(y).
\end{eqnarray*}

It is enough to show that $\int_{d(x,y)\geq t^{1/m}}|k_t(x,y)|d\mu(x)$ is uniformly bounded for $t\geq 1$.

The identity $\Delta ^{-\alpha}=\int_0^\infty e^{-s\Delta }\frac{ds}{s^{1-\alpha}}$ (we ignore the constant here) gives us
\[
T(I-e^{-t\Delta })=\int_0^\infty \nabla e^{-(s+1)\Delta }(I-e^{-t\Delta })\frac{ds}{s^{1-\alpha}},
\]
that is,
\begin{eqnarray*}
k_t(x,y)
&=& 
\int_0^\infty \left(\nabla p_{s+1}(x,y)-\nabla p_{s+t+1}(x,y)\right)\frac{ds}{s^{1-\alpha}}
\\&=&
\int_0^\infty \left(\frac{1}{s^{1-\alpha}}-\frac{1_{\{s>t \}}}{(s-t)^{1-\alpha}}\right)\nabla p_{s+1}(x,y)ds.
\end{eqnarray*}

Thus by using the estimate (\ref{we}), we have
\begin{eqnarray*}
&&\int_{d(x,y)\geq t^{1/m}}\left| k_t(x,y)\right| d\mu(x)
\\&=&
\int_{d(x,y)\geq t^{1/m}}\left|\int_0^\infty \left(\frac 1{s^{1-\alpha}}-
\frac{1_{\{s>t \}}}{(s-t)^{1-\alpha}}\right)\nabla p_{s+1}(x,y)ds \right|d\mu(x)
\\&\leq &
\int_0^\infty \left|\frac 1{s^{1-\alpha}}-\frac{1_{\{s>t \}}}{(s-t)^{1-\alpha}}\right|
\cdot \int_{d(x,y)\geq t^{1/m}}\left|\nabla p_{s+1}(x,y)\right|d\mu(x)ds
\\&\lesssim & 
\int_0^\infty \left|\frac 1{s^{1-\alpha}}-\frac{1_{\{s>t\}}}{(s-t)^{1-\alpha}}\right| (s+1)^{-\alpha}e^{-c\left(\frac{t}{s+1}\right)^{1/(m-1)}}ds
\\&=& 
\left(\int_0^1 +\int_1^t \right)\frac{1}{s^{1-\alpha}(s+1)^{\alpha}} e^{-c\left(\frac{t}{s+1}\right)^{1/(m-1)}}ds
\\&&+ 
\int_t^\infty \left|\frac 1{s^{1-\alpha}}-\frac{1}{(s-t)^{1-\alpha}}\right|(s+1)^{-\alpha} e^{-c\left(\frac{t}{s+1}\right)^{1/(m-1)}}ds
\\&=& K_1+K_2+K_3.
\end{eqnarray*}

In fact, $K_1,K_2,K_3$ are uniformly bounded:
\[
K_1\leq \int_0^1 s^{\alpha-1}ds<\infty;
\]

Since $s+1 \simeq s$ for $s>1$ and we can dominate the $e^{-x}$ by $C x^{-c}$ for any fixed $c>0$, we have
\[
K_2\leq \int_1^t e^{-c'\left(\frac{t}{s}\right)^{1/(m-1)}}\frac{ds}{s} \leq C\int_1^t \br{\frac{s}{t}}^c \frac{ds}{s}<\infty;
\]

For $K_3$, 
\begin{eqnarray*}
K_3&\leq &\int_t^\infty \abs{\frac 1{s^{1-\alpha}}-\frac {1}{(s-t)^{1-\alpha}}}s^{-\alpha}ds
\\ & =&
\int_0^\infty \abs{\frac 1{(u+1)^{1-\alpha}}-\frac {1}{u^{1-\alpha}}}(u+1)^{-\alpha}du
\\&\leq & 
\int_0^1\br{\frac 1{(u+1)}+\frac {1}{u^{1-\alpha}}}(u+1)^{-\alpha}du+\int_1^\infty \frac1{(u+1)u^{1-\alpha}}du
\\&\leq & 
\int_0^1 \frac 2{u^{1-\alpha}}du+\int_1^\infty \frac1{u^{2-\alpha}}du <\infty.
\end{eqnarray*}
\medskip
Note that we get the second line by changing variable with $u=\frac{s}{t}-1$.

\noindent{\bf Case 2: } It remains to show
\[
\mu(\{x: |T \sum_{i\in \mathcal {C}_2}b_i(x)|>\lambda/ 4\})\leq C\lambda^{-1}\Vert f\Vert_1.
\]

We repeat the argument as Case 1. Still from (\ref{divide}), we have 
\begin{eqnarray*}
\mu(\{x: |T \sum_{i\in \mathcal {C}_2}b_i(x)|>\lambda/ 4\}) 
& \leq & \mu(\{x: |T \sum_{i\in \mathcal {C}_2}e^{-t_i\Delta }b_i(x)|>\lambda/ 8\})
\\ &&+ \mu(\{x: |T \sum_{i\in \mathcal {C}_2}(I-e^{-t_i\Delta })b_i(x)|>\lambda/ 8\}).
\end{eqnarray*}

By using the $L^2$ boundedness of $T$, the same duality argument in Case 1 yields
\[
\mu(\{x: |T \sum_{i\in \mathcal {C}_2}e^{-t_i\Delta }b_i(x)|>\lambda/ 8\}) \leq C \lambda^{-1} \Vert f\Vert_1.
\]

For the estimate of $\mu(\{x: |T\displaystyle \sum_{i\in \mathcal {C}_2}(I-e^{-t_i\Delta })b_i(x)|>\lambda/ 8\})$, it suffices to show 
that 
$\int_{d(x,y)\geq  t^{1/2}}\abs{k_t(x,y)}d\mu(x)$ is finite and does not depend on $t<1$. In fact,
\begin{eqnarray*}
&&\int_{d(x,y)\geq t^{1/2}}\left|k_t(x,y)\right|d\mu(x)
\\ &\leq& 
\int_0^\infty \abs{\frac 1{s^{1-\alpha}}-\frac{1_{\{s>t\}}}{(s-t)^{1-\alpha}}}
\cdot \int_{d(x,y)\geq t^{1/2}}\abs{\nabla p_{s+1}(x,y)}d\mu(x)ds
\\ &\lesssim& 
\int_0^\infty \abs{\frac 1{s^{1-\alpha}}-\frac{1_{\{s>t\}}}{(s-t)^{1-\alpha}}} (s+1)^{-\alpha} e^{-c\br{\frac{t^{m/2}}{s+1}}^{1/(m-1)}}ds
\\ &=& 
\int_0^t \frac{1}{s^{1-\alpha} (s+1)^{\alpha}} e^{-c\br{\frac{t^{m/2}}{s+1}}^{1/(m-1)}}ds
\\ &&
+\int_t^\infty \abs{\frac 1{s^{1-\alpha}}-\frac{1}{(s-t)^{1-\alpha}}} (s+1)^{-\alpha} e^{-c\br{\frac{t^{m/2}}{s+1}}^{1/(m-1)}}ds
\\&:=& 
K'_1+K'_2.
\end{eqnarray*}

Because $t<1$, thus $K'_1<K_1$ converges.

We can estimate $K'_2$ in the same way as for $K_3$ and get a bound that does not depend on $t$.
\end{proof}

\bigskip

%%%%%%%%%%%%%%%%%%%%%%%%%%%%%%%%%%%%%%%%%%%%%%%%%%%%%%%%%%%%%%
{\bf Acknowledgements:} I would like to thank Thierry Coulhon for many helpful discussions and for his constant encouragement. 
Thanks are also due to Fr\'ed\'eric Bernicot, Shibing Chen and the referee for useful suggestions and remarks. The results in this paper are part of my PhD thesis in preparation in  
cotutelle between the Laboratoire de Math\'ematiques, Orsay and the Mathematical Science Institute, Canberra under the 
supervision of Pascal Auscher and Thierry Coulhon.

\end{document}